\documentclass[12pt]{amsart}
\usepackage{thmtools}
\usepackage{amsfonts,amsmath,amssymb}
\usepackage{amssymb}
\usepackage{amsthm}  
\usepackage{amsmath} 
\usepackage{amscd} 
\usepackage{float}
\usepackage[all]{xypic}
\usepackage{url}
\usepackage{hyperref}
\usepackage{color}
\hypersetup{linktocpage}

\usepackage{amsmath,amssymb,amsthm,enumerate}

\newcommand{\C}{\mathbb{C}}

\newcommand{\todo}[1]{}

\newlength{\extendaxesby}

\newtheorem{thm}{Theorem}
\newtheorem{theorem}{Theorem}

\newtheorem{proposition}[thm]{Proposition}

\theoremstyle{definition}

\newtheorem{definition}[thm]{Definition}

\renewcommand{\Re}{\mathop{\rm Re}\nolimits}
\renewcommand{\Im}{\mathop{\rm Im}\nolimits}

\def\beq{\begin{equation}}
\def\eeq{\end{equation}}

\newcommand{\CC}[1]{\mathbb{C}^{#1}}

\newcommand{\lr}{\longrightarrow}

\numberwithin{equation}{section}

\def\Label#1{\label{#1}{\sf (#1)}~}

\usepackage{multirow}

\newcommand{\om}{\omega}

\newcommand{\ba}{\mathcal{G}}

\newcommand{\fb}{\mathfrak b}

\newcommand{\bz}{\bar z}

\newcommand{\bl}{\bar l}
\newcommand{\bP}{\bar P}
\newcommand{\bX}{\bar X}

\numberwithin{table}{section}

\begin{document}
\title[The jet transcendence degree of a real hypersurface]{The jet transcendence degree of a real hypersurface  and Huang-Ji-Yau Conjecture}
\author{Jan Gregorovi\v{c}}
\address{Jan Gregorovi\v{c}, Department of Mathematics, Faculty of Science, University of Ostrava, 701 03 Ostrava, Czech Republic, and Institute of Discrete Mathematics and Geometry, TU Vienna, Wiedner Hauptstrasse 8-10/104, 1040 Vienna, Austria}\email{ jan.gregorovic@seznam.cz}
\author{Ilya Kossovskiy}
\address{Department of Mathematics, Sustech University, China//
Department of Mathematics, Masaryk University, Brno, Czechia//
Institute of Discrete Mathematics and Geometry, TU Vienna, Vienna, Austria
}
\email{kossovskiyi@math.muni.cz}
\subjclass[2010]{}
 \thanks{The authors were supported by the GACR grant 22-15012J and by Austrian Science Fund (FWF) grant P34369}
\maketitle
\begin{abstract}
We investigate the problem of holomorphic algebraizibility for real hypersurfaces in complex space. We introduce a new invariant of a (real-analytic) Levi-nondegenerate hypersurface called {\em the jet transcendence degree}. Using this invariant, we solve in the negative the Conjecture of Huang, Ji and Yau on the algabraizability of real hypersurfaces with algebraic syzygies.
\end{abstract}

\tableofcontents

\section{Introduction}

A real hypersurface $M\subset\CC{n+1}$ is called {\em algebraizable} near $p\in M$, if it is CR-diffeomorphic, locally near $p$, to an algebraic real hypersurface. The study of algebraizability of real-analytic hypersurfaces has a long history. We refer the reader here to  e.g. the work of Forstneric \cite{forstnericalg} and references therein. This study is motivated by the fact that much more tools and results for studying holomorphic maps of real submanifolds and domains in complex space are available in the algebraic setting (in particular, all CR-maps in the algebraic setting are algebraic as well, under very mild assumptions on the CR-manifolds; we refer here to the original paper \cite{websteralg} of Webster on the subject, and a more recent paper \cite{zaitsevalg} of Zaitsev and references therein). At the same time, being algebraizable is a very rare property of an analytic CR-manifold, i.e. "most"\,  real-analytic hypersurfaces are {\em not} algebraizable, see Forstneric \cite{forstnericalg} for a general argument here and Zaitsev \cite{zaitsevquad} for more a more explicit (differential-algebraic) description of algebraizable hypersurfaces. 

The problem of detecting algebraizable CR-manifolds among general analytic ones remains widely open, {\em even in the case of Levi-nondegenerate hypersurfaces}. A certain advance in this direction is due to Huang, Ji and Yau who obtained in \cite{HJY01} an obstruction for submanifold $M$ to be algebraizable. They observed the following:

\smallskip

{\it Let $P_1,\dots,P_n$ be a functionally independent set of holomorphic invariants of the Segre family corresponding to an analytic CR-manifold $M\subset \C^N$. If $M$ is algebraizable, then for each further invariant $P$ functionally dependent on $P_1,\dots,P_n$, there is a polynomial $R_P$ such that $R_P(P,P_1,\dots,P_n)=0$.}

\smallskip

In other words, {\em syzygies  (functional relations between invariants) associated with an algebraizable CR-manifold are necessarily algebraic}.

Let us recall that the holomorphic invariants of the Segre family of a Levi-nondegenerate hypersurface $M\subset\CC{N}$ are the structure functions of a Cartan connection $\omega_{\C}$ of \cite{C75,CM74} on a holomorphic $(N^2+2N)$-dimensional bundle $\mathcal{Y}$ over the complexification of $M$. The curvature invariants of $\omega_{\C}$ and the corresponding derived invariants (as functions on  $\mathcal{Y}$) then allow to work out the necessary condition of Huang, Ji and Yau for algebraizibility. Arguing in this manner, the authors in \cite{HJY01} construct an {\em explicit} real-analytic hypersurface $M\subset\CC{2}$ for which the syzygies are not algebraic, so that $M$ is not algebraizable as a corollary. They conclude by the following
\smallskip

\noindent{\bf Conjecture} (Huang-Ji-Yau, 2007, \cite{HJY01}). Let $M\subset\CC{N}$ be a real-analytic Levi-nondegenerate hypersyrface for which all the syzygies arising from the complete system of Cartan-Chern-Moser invariants are algebraic. Then $M$ is (locally) algebraizable.

\smallskip

The main goal of this paper is to resolve the above Conjecture in the negative. We do so by providing  a different than in \cite{HJY01} obstruction for algebraizibility, which is the positivity of the {\em jet transcendence degree} of a Levi-nondegenerate hypersurface. The respective (integer-valued) invariant of  a hypersurface is introduced and studied in Section 3. We show that {\em for an algebraizable hypersurface, its jet transcendence degree  vanishes} (\autoref{obstruct}). At the same time, we are able to present a hypersurface in $\CC{2}$, for which the  jet transcendence degree is {\em positive}, while all the Cartan-Chern-Moser syzygies are {\em algebraic}. This resolve the Conjecture in the negative.

\begin{theorem}[Main Theorem]
Suppose a rigid real-analytic hypersurface 
\[\Re w=H(z,\bz)\]
 in $\C^2$ satisfies that $H_{z \bz}$ is algebraic and everywhere nonzero. Then all its Cartan-Chern-Moser syzygies  are algebraic and thus the obstructions of Huang, Ji and Yau do not occur.

However, if the rigid hypersurface in addition satisfies that $H_{z}(z, \bar z)$ is non-algebraic in $\bar z$, $H_{zz}$ is algebraic in $\bar z$, and $H_{zz \bar z}$ is not identical $0$, then its jet transcendence degree is positive, and thus the hypersurface is not holomorphically equivalent to an algebraic hypersurface in $\C^2$.
\end{theorem}

A particular example of a hypersurface satisfying all conditions of the Main Theorem is given by defining equation
\begin{align}\label{exmp1}
\Re w=z \bz+z \arctan(\bz)+\arctan(z)\bz.
\end{align}

Let us remark that the nonalgebraizable rigid submanifold 
$$\Im w=\exp(|z|^2)-1$$ 
of Huang, Ji and Yau has the a positive jet transcendence degree.

\bigskip

We finish the Introduction by the following natural leftover question:

\smallskip

\noindent{\bf Problem}. Let $M\subset\CC{n+1}$ be a real-analytic Levi-nondegenerate hypersyrface for which all the syzygies arising from the complete system of Cartan-Chern-Moser invariants are algebraic and the jet transcendence degree is vanishing. Is $M$ (locally) algebraizable?

\smallskip
\section{Preliminaries}

\subsection{The method of associated differential equations} It was observed by Cartan \cite{C32} and Segre \cite{segre} (see also Webster \cite{websteralg}) that the geometry of a real hypersurface in $\CC{2}$ parallels that of a second order ODE
 \begin{equation}\Label{wzz}
 w\rq{}\rq{}=\Phi(z,w,w\rq{}).
 \end{equation}
More generally,  the geometry of a real hypersurface in $\CC{n+1},\,n\geq 1$, parallels that of a complete second order system of PDEs
\begin{equation}\Label{wzkzl}
w_{z_kz_l}=\Phi_{kl}(z_1,...,z_n,w,w_{z_1},...,w_{z_n}),\quad \Phi_{kl}=\Phi_{lk},\quad k,l=1,...,n.
\end{equation}
  Moreover, {\em in the real-analytic case} this parallel becomes algorithmic by using  the Segre family of a real hypersurface. With any real-analytic Levi-nondegenerate hypersurface $M\subset\CC{n+1},\,n\geq 1$ one can uniquely associate a holomorphic ODE \eqref{wzz} ($n=1$) or a holomorphic PDE system \eqref{wzkzl} ($n\geq 2$). The Segre family of $M$ plays a role of a mediator between the hypersurface and the associated differential equations.  

  The associated differential equation procedure is particularly clear in the case of a Levi-nondegenerate hypersurface in $\CC{2}$. In this case the Segre family is a 2-parameter anti-holomorphic family of  holomorphic curves. It then follows from standard ODE theory that there exists an unique ODE \eqref{wzz}, for which the Segre varieties are precisely the graphs of solutions. This ODE is called \it the associated ODE. \rm

In the general case, both right hand sides in \eqref{wzz},\eqref{wzkzl} appear as functions determining the $2$-jet of a Segre variety as an analytic function of the $1$-jet. More explicitly, we denote the coordinates in $\CC{n+1}$ by 
$$
	(z,w)=(z_1, \ldots, z_n,w).
$$ 
Let then fix $M\subset\CC{n+1}$ to be a smooth real-analytic
hypersurface, passing through the origin, and choose a small neighborhood $U$
 of the origin. In this case
we associate a complete second order system of holomorphic PDEs to $M$,
which is uniquely determined by the condition that the differential equations are satisfied by all the
graphing functions $h(z,\zeta) = w(z)$ of the
Segre family $\{Q_\zeta\}_{\zeta\in U}$ of $M$ in a
neighbourhood of the origin.
To be more explicit we consider the
so-called {\em  complex defining
 equation } (see, e.g., \cite{ber})\,
$w=\rho(z,\bar z,\bar w)$ \, of $M$ near the origin, which one
obtains by substituting $u=\frac{1}{2}(w+\bar
w),\,v=\frac{1}{2i}(w-\bar w)$ into the real defining equation and
applying the holomorphic implicit function theorem.
 The Segre
variety $Q_p$ of a point 
$$x=(a,b)\in U,\,a\in\CC{n},\,b\in\CC{}$$ 
is  now given
as the graph
\begin{equation} \Label{segredf}w (z)=\rho(z,\bar a,\bar b). \end{equation}
Differentiating \eqref{segredf} we obtain
\begin{equation}\Label{segreder} 
	w_{z_j}=\rho_{z_j}(z,\bar a,\bar b),
	\quad
	j=1,\ldots,n. 
\end{equation}
Considering \eqref{segredf} and \eqref{segreder}  as a holomorphic
system of equations with the unknowns $\bar a,\bar b$, 
in view the Levi-nondegeneracy of $M$,
an
application of the implicit function theorem yields holomorphic functions
 $A_1,...,A_n, B$ such that
 \eqref{segredf} and \eqref{segreder} are solved by
\beq \label{AB}
	\bar a_j=A_j(z,w,w\rq{}),\quad
	\bar b=B(z,w,w\rq{}),
\eeq
where we write
$$
	w\rq{}= (w_{z_1},  \ldots, w_{z_n}).
$$
The implicit function theorem applies here because the
Jacobian of the system coincides with the Levi determinant of $M$
for $(z,w)\in M$ (\cite{ber}). Differentiating \eqref{segredf} twice
and substituting the above solution for $\bar a,\bar b$ finally
yields
\begin{equation}\Label{segreder2}
w_{z_kz_l}=\rho_{z_kz_l}(z,A(z,w,w\rq{}),
B(z,w,w\rq{}))=:\Phi_{kl}(z,w,w\rq{}),
\quad
k,l=1, \ldots, n,
\end{equation}
or, more invariantly,
\begin{equation}\Label{segreder2'}
	j^2_{(z,w)} Q_x = \Phi(x, j^1_{(z,w)} Q_x).
\end{equation}
Now \eqref{segreder2} is the desired complete system of holomorphic second order PDEs
denoted by $\mathcal E = \mathcal{E}(M)$.
 \begin{definition}\Label{PDEdef}
 We call $\mathcal E = \mathcal{E}(M)$  \it the system of PDEs 
 associated with $M$. \rm  
  We also regard the collection   $\{\Phi_{ij}\}_{i,j=1}^n$ 
  as {\em the PDE system defining the CR structure} of a Levi-nondegenerate hypersurface $M$.
\end{definition}

The holomorphic invariance of the Segre family  implies the following key property: {\em every local biholomorphism $H:\,(M,0)\lr(\tilde M,0)$ locally transforms the associated ODEs $\mathcal E(M),\mathcal E(\tilde M)$ into each other.}

For more on the method of associated differential equations, see e.g. Sukhov \cite{sukhov1,sukhov2}, Kossovskiy-Shafikov \cite{divergence,nonminimalODE}, Kossovskiy-Lamel \cite{nonanalytic}.

\subsection{Invariants of Segre families  and the obstruction of Huang, Ji and Yau}

Let us discuss the obstruction of Huang, Ji and Yau, in detail.

The invariants of Segre family of Levi-nondegenerate hypersurfaces in $\C^2$ were first obtained by Tresse \cite{T96} and it was observed by Cartan \cite{C32} that they are induced by the CR invariants of the real submanifold $M$, but they are not sufficient for the solution of the equivalence problem of real submanifolds $M$ in $\C^2$. A particular example of this situation was constructed by Faran \cite{F80}. In fact, the relation between the invariants of Segre families and CR invariants is provided by complexification.

The first construction of CR invariants solving the equivalence problem of strongly pseudo--convex real submanifolds $M$ in $\C^2$ was obtained by Cartan in \cite{C32}, further solutions were obtained by Tanaka in \cite{T62,T75} or Chern and Moser in \cite{CM74}. These were generalized to explicit solutions of equivalence problem of Segre families by Chern \cite{C75} and Faran \cite{F80}. What is important for the solution of the equivalence problems is:

\begin{enumerate}
\item In the case of strongly pseudo--convex real submanifolds in $\C^2$, to construct a fiber bundle $\ba$ over $M$ together with a $\frak{su}(1,2)$--valued absolute parallelism $\om$ on $\ba$, i.e., $\om: T\ba\to \frak{su}(1,2)$ provides at each point $u\in \ba$ a linear isomorphism of the tangent space $T_x\ba$ with the simple Lie algebra $\frak{su}(1,2)$.
\item In the case of the corresponding Segre families, to construct a holomorphic fiber bundle $\mathcal{Y}=\ba_\C$ over the complexification $M_\C$ of $M$ together with a holomorphic $\frak{sl}(3,\C)$--valued absolute parallelism $\om_\C$ on $\ba_\C$.
\end{enumerate}

The problem of biholomorphic equivalence problem is then equivalent to equivalence problem of absolute parallelisms for which there is a uniform and well established theory, c.f. Olver \cite{O95}. In particular, for an general absolute parallelism $\omega$ on a fiber bundle $\ba$ over $M$,

\begin{itemize}
\item firstly, from the structure equation $\Omega=d\omega+[\omega,\omega]$, one obtains zero order invariants, and then
\item by differentiation of the other invariants in direction of the vector fields $\omega(\xi)=\rm{const},$ one obtains higher order (depending on number differentiations) derived invariants.
\end{itemize}

The image of the structure bundle $\ba$ in the space of all invariants of order less than $k$ is usually called the classifying set of order $k$. The main point is that the classifying sets do not depend on the choice of coordinates on $M$ and $\ba$ and two absolute parallelisms $\omega$ are (locally) equivalent if and only if they have the same (overlapping) classifying sets (of sufficiently large order $k$).

Let us emphasize that the order of the derived invariants should not be confused with the (weighted) order of the invariants in the derivations of the defining equation of the submanifold $M$/Segre family. Of course, the Bianchi and Ricci identities provide functional and differential dependence between some of the (derived) invariants. On the other hand, if one does not obtain new functionally independent invariants, then (in regular situation) one has found a maximal set of functionally independent invariants and all others are functionally dependent. 

Of course, one does not have to consider all of the  invariants (of order less than $k$), and relation $R_P(P,P_1,\dots,P_N)=0$ between the functionally dependent invariants can be seen as a possible collection of defining equations for the classifying set. Therefore, the obstructions of algebraizibility of Huang, Ji and Yau can be reformulated as follows:

\medskip

{\em If a Levi-nondegenerate real-analytic hypersurface $M\subset \C^2$ is algebraizable, then all classifying sets of $\om_\C$ are open subsets of (complex) algebraic varieties.}

\medskip

Certain (significant) modifictions shall be made when formulating analogous obstructions for classifying sets of $\om$ itself, because in this case the varieties would be {\em real algebraic}.

\subsection{Cartan connections}

In this section, we discuss the solution of the equivalence problem for strictly pseudoconvex CR-hypersurfaces in $\CC{2}$ (and their Segre families) that employ {\em Cartan connections}, i.e., the structure bundle $\ba$ is a principal fiber bundle and the absolute parallelism $\om$ is equivariant w.r.t. to the structure group. In this case, the structure bundle $\mathcal{Y}=\ba_\C$ and the Cartan connection $\om_\C$ can be directly obtained by the complexification. 
We use the construction described in \v Cap and Slovak in \cite[Section 3.1, Theorem 3.1.14]{P}, because this unlike the construction in  \cite{C75,CM74,HJY01} does not require finding the equation  $w\rq{}\rq{}=\Phi(z,w,w\rq{})$. Moreover, since the relations $R_P(P,P_1,\dots,P_n)=0$ are independent on the choice of coordinates, it is simpler to use this Cartan connection for checking the the obstructions of algebraizibility of Huang, Ji and Yau.

To recall the structure groups involved, it is sufficient to recall the (maximally symmetric) models. A (local) model of such real hypersurface is a quadric 
\[\Im w=z\bz\]
 in $\C^2$ that can be identified with an open subset in a (global) model $PSU(1,2)/B$, which is the null cone in $\C P^2$, i.e., space of null lines in the standard representation of $SU(1,2)$ on $\C^3$ and $B$ is a stabilizer of a null line (and a Borel subgroup of $PSU(1,2)$). The model of the corresponding Segre family is then identified with the homogeneous space $PGL(3,\C)/B_\C$ and corresponds to the trivial second--order ODE \[w\rq{}\rq{}=0.\]

Let us summarize that a Cartan geometries $(\ba\to M,\om)$  of type $(PSU(1,2),B)$ and $(\ba_\C\to M_\C,\om_\C)$  of type $(PGL(3,\C),B_\C)$ are
\begin{itemize}
\item a (right) principal $B$--bundle  $\ba$ over $M$ and $B_\C$--bundle  $\mathcal{Y}=\ba_\C$ over $M_\C$,
\item a $B$--equivariant $\frak{su}(1,2)$--valued absolute parallelism $\om$ on $\ba$ and a $B_\C$--equivariant $\frak{sl}(3,\C)$--valued absolute parallelism $\om_\C$ on $\ba_\C$ that both reproduce the fundamental vector fields of the right action of the structure groups $B$ and $B_\C$ on $\ba$ and $\ba_\C$, respectively.
\end{itemize}

We provide the explicit construction of $\om$ and $\om_\C$ in the next section. Here we just mention that the construction involves a choice of normalization conditions  \cite[Section 3.1.12]{P} that we will not recall here. 

\section{The jet transcendence degree of a Levi-nondegenerate hypersurface}

In this section, we observe a (previously unknown) invariant of a Levi-nondegenerate real hypersurface, which appears useful for studying the algebraizability problem and, in particular, allows to resolve the Huang-Ji-Yau Conjecture. 


The {\em transcendence degree} of a holomorphic at the origin (vector) function $$F(x,y)=\bigl(f_1(x,y),...,f_s(x,y)\bigr),\,\,(x,y)\in\CC{m}\times\CC{n}$$ in the (vector) variable $y$  can be defined as follows. Consider the algebra of germs at the origin of functions $P(x,y)$ polynomial in $y$ and holomorphic in $x$, and the respective  field of fractions $\mathcal K$. For a function $F(x,y)$ as above, we consider now the extension $\mathcal L/\mathcal K$ obtained by adding the components of  $F$ to $\mathcal K$. Then the  {\em transcendence degree} of $F$ in $y$ is the transcendence degree of the extension $\mathcal L/\mathcal K$. The latter, we recall, is the cardinality of the maximal algebraically independent subset $\mathcal S$ in    $\mathcal L/\mathcal K$. In particular, an extension is algebraic if and only if its transcendence degree is $0$. In accordance with that, an {\em algebraic in $y$} function $F(x,y)$ is one with the zero transdencence degree. For more on this, we refer to e.g. Lang \cite{lang}.


For a real-analytic Levi-nondegenerate hypersurface $M\subset\CC{n+1},\,n\geq 1$ near its distinguished point $0\in M$, we consider the invariantly associated system of differential equations $\mathcal E(M)$ (given by either \eqref{wzz} or \eqref{wzkzl}). 

\begin{definition} 
The {\em jet transcendence degree} of $M$ is the transcendence degree in $$w_z=\bigl(w_{z_1},...,w_{z_n}\bigr)$$ of the function $\Phi$ as in \eqref{wzz} ($n=1$) or \eqref{wzkzl} ($n>1$). 
\end{definition}

We shall prove the invariance of the jet transcendence degree. For that,  we shall subject $\Phi$ to a local origin-preserving biholomorphis 
$$H(z,w)=(f_1,...,f_n,g)=\bigl(f(z,w),g(z,w)\bigr).$$ Then the transformation rule for $\Phi$ under $H$ is obtain by lifting $H$ to a (local) self-map $H^{(2)}$ of the bundle $J^{2,n}$ of $2$-jets of real hypersurfaces. Considering \eqref{wzz} or \eqref{wzkzl}) as a submanifold $\mathcal E$ of  $J^{2,n}$, we subject $\mathcal E$ to the transformation $H^{(2)}$, and the target submanifold gives the desired function $\tilde\Phi$ associated with the image of $M$.

To control the change of $\Phi$ here, we need more information on the prolonged map $H^{(2)}$. Let us set
$$w_j:=w_{z_j},\,\,w_{kl}:=w_{z_kz_l},\,\,w':=\left\{w_{k}\right\}_{k=1}^n,\quad w'':=\left\{w_{kl}\right\}_{k,l=1}^n.$$
Then $(z,w,w',w'')$ can be treated as a tuple of local coordinates in $J^{2,n}$.
Following now the procedure in e.g. \cite{bluman}, we see that the prolonged map 
$$H^{(2)}(z,w,w',w'')=\bigl(f(z,w),g(z,w),g^{1}(z,w,w'),g^{2}(z,w,w',w'')\bigr)$$ 

has the following properties:

\smallskip

\noindent (i) the component $g^{1}(z,w,w')$ a is birational mapping $R(z,w,w')$ in $w'$ with coefficients depending polynomially on the first derivatives of $H$;

\smallskip

\noindent (ii) the component $g^{2}(z,w,w',w'')$ has the form $$A(z,w,w_z)w_{zz}+B(z,w,w_z),$$
where $A,B$ are rational in $w'$ with coefficients depending polynomially on the first and second derivatives of $H$, and $\mbox{det}\,A\not\equiv 0$.

\smallskip

More specifically, we make use of the {\em total derivations}
$$D_k:=\frac{\partial}{\partial z_k}+w_k\frac{\partial}{\partial w}+\sum_{l=1}^n w_{kl}\frac{\partial}{\partial w_l},$$
set
$$C:=\bigl(D_kf_l\bigr)_{k,l=1}^n$$
(the invertibility of $H$ implies here $\mbox{det}\,C\not\equiv 0$),
and then
$$g^{1}=C^{-1}\cdot Dg,\quad Dg:=(D_1g,...,D_ng)^T.$$
Further, we have
$$g^{2}=C^{-1}\cdot Dg^1,\quad Dg^1:=\bigl(D_kg^1_l\bigr)_{k,l=1}^n.$$

Gathering the information important for our particular purposes, we conclude that the transformation rule for $\Phi$ has the following form:

\begin{equation}\Label{transrule}
\Phi(z,w,w')=Q(z,w,w')\cdot\tilde\Phi\bigl(f(z,w),g(z,w),R(z,w,w')\bigr)+S(z,w,w'),
\end{equation}
where $Q,R,S$ are all rational functions in $w'$ (and $Q\not\equiv 0$) with coefficients being polynomials depending on derivatives of $f,g$ of order $1$ and $2$. 
For the convenience of the reader, we provide here the explicit formula for the case $n=1$:
\begin{multline}\label{trule}
\Phi(z,w,w\rq{})=\frac{1}{J}\bigg((  f_z+w\rq{}  f_w)^3\cdot\tilde\Phi\left(  f(z,w),  g(z,w),\frac{  g_z+w\rq{}  g_w}{  f_z+w\rq{}  f_w}\right)+\\
I_0(z,w)+I_1(z,w)w\rq{}+I_2(z,w)(w\rq{})^2+I_3(z,w)(w\rq{})^3\bigg),
\end{multline}
where $J:=  f_z  g_w-  f_w  g_z$ is the Jacobian determinant of the transformation and
\beq\begin{aligned}
I_0 &=  g_z  f_{zz}-  f_z  g_{zz}\\
I_1 &=  g_w  f_{zz}-  f_w  g_{zz}-2  f_z  g_{zw}+2  g_z  f_{zw}\\
I_2 &=  g_{z}  f_{ww}-  f_z  g_{ww}-2  f_w  g_{zw}+2  g_w  f_{zw}\\
I_3 &=  g_w  f_{ww}-  f_w  g_{ww}.
\end{aligned}\eeq 

The explicit nature of the transformation rule \eqref{transrule} implies the invariance under local biholomorphisms of algebraic relations between the components of $\Phi$ with coefficiants from the field $\mathcal K$ as above (we take here into account the birationality of $R$), and we immediately obtain:
\begin{proposition}\Label{invar}
The jet transcendence degree of a Levi-nondegenerate hypersurface is invariant under local biholomorphisms. 
\end{proposition} 

The following (simple) proposition illuminates the connection between the jet transcendence degree and the algebraizability. 

\begin{proposition}\label{obstruct}
If a real-analytic submanifold $M\subset \C^{n+1}$ is algebraizable, then $M$ has the zero transcendence degree, i.e. the defining function $\Phi$ of the associated system \eqref{wzz} or \eqref{wzkzl} is algebraic in $w\rq{}=(w_{z_1},...,w_{z_n})$.
\end{proposition}
\begin{proof}
We first observe that, for an algebraic target, the associated function $\tilde\Phi$ has a finite transcendence degree in $w'$, i.e. is algebraic in $w'$.  (In fact, such a $\tilde\Phi$  is algebraic in all its variables). This follows directly from the explicit procedure for constructing the associated function $\Phi$ described in Section 2.1 and elementary properties of algebraic functions. We now employ \autoref{invar} to conclude that the original associated function $\Phi$ also has the zero transcendence degree in $w'$, i.e. $M$ has the zero transcendence degree, as required.
\end{proof}

\section{Explicit construction of invariants of Segre families}\label{Expl}

\subsection{Construction of the Cartan connection}

The main advantage of the use of Cartan connection in the construction of invariants is that it is sufficient to choose a (local) section $s: M_\C\to \ba_\C$ to pullback the computations to $M_\C$ and then use the structure group to obtain the final result on $\mathcal{Y}=\ba_{\C}.$ We describe the relation of what follows with the construction in the CR case is in Appendix \ref{appB}.

We start with a rigid real-analytichypersurface $\Re w=H(z,\bz)$ in $\C^2$ satisfies that $H_{z \bz}$ is everywhere nonzero. In Appendix \ref{Ap1}, we discuss, how the results in this section can be used in the case of general defining function. We start with construction of the pullback $s^*\om_\C$ of the Cartan connection $\om_\C$.

We use the complex coordinates $(z,\bz,v,\partial_z,\partial_{\bz},\partial_v)$ on $TM_\C$ and we consider the vector fields
\begin{align*}
X&=\partial_z-iH_z\partial_v,\\
\bX&=\partial_{\bz}+iH_{\bz}\partial_v,\\
Y&=-2iH_{z\bz}\partial_v,
\end{align*}
where we omit the arguments of $H$ and use the subscript for the partial derivatives. Observe, that $X,\bar X$ generate the complexification of the complex tangent space $TM\cap i(TM)$ and that $[X,\bar X]=-Y$, which provides the Heisenberg Lie algebra complement to $\fb_\C$ in $\frak{sl}(3,\C)$ as required in \cite[Section 3.1]{P}. Therefore, we use the corresponding coframe $(j,l,\bl)$ dual to the frame $(Y,X,\bX)$ and compute that
\begin{align*}
dj &=P j\wedge l+\bP j\wedge \bl+l\wedge \bl,\\
 dl &=0,\\
  d\bl&=0,\\
P\cdot Y&:=-[X,Y]=\frac{H_{zz \bz}}{H_{z \bz}}\cdot Y.
\end{align*}
Note that a consequence of Jacobi identity is $P_{\bX}=\bP_X$, i.e., $P_{\bX}$ is real on the real form $\ba\subset \ba_\C$, where the subscript is used for the derivatives in the direction of the vector fields. For general defining equations the derivations in direction of vector fields of $P$ would depend on the ordering, but our assumptions imply:
\begin{align*}
P_{X}&=P_{z}=\frac{H_{zzz\bz}H_{z \bz}-H_{zz \bz}H_{zz\bz}}{H_{z \bz}^2},\\
P_{\bX}&=P_{\bz}=\frac{H_{zz \bz\bz}H_{z \bz}-H_{zz \bz}H_{z\bz \bz}}{H_{z \bz}^2},\\
P_Y&=P_v=0,
\end{align*}
and thus the nonholomic frame $(Y,X,\bX)$ acts on $P$ in the same way as the holonomic frame $(\partial_z,\partial_{\bz},\partial_v)$. However, we present the formula for the pullback $s^*\om_\C$ using the nonholomic frame $(Y,X,\bX)$, because such formulas also applies for general strongly pseudo--convex hypersurfaces in $\C^2$. To obtain the following formula for $s^*\om_\C$, we have to chosen the section $s$ in a particular way and then checked the normalization condition of \cite[Section 3.1.12]{P} using the Differential geometry package in Maple:
{\tiny
\begin{align*}
j\left[ \begin{array}{ccc}
-\frac{1}{12}P_{\bX}& \frac16(PP_{\bX}+2P_{Y}-P_{X\bX})&\frac16(PP_{\bX\bX}-P\bP P_{\bX}+\frac{11}{8}P_{\bX}^2-P_{X\bX\bX}+\bP P_{X\bX}-\frac12 P_{Y\bX}) \\
0&\frac{1}{6}P_{\bX}&-\frac16(\bP P_{\bX}-2\bP_{Y}-P_{\bX\bX})\\
1& 0& -\frac{1}{12}P_{\bX}
\end{array}\right]\\
+l\left[ \begin{array}{ccc}
-\frac{2}{3}P&0&\frac1{12}(PP_{\bX}-4P_{Y}-P_{X\bX}) \\
1&\frac{1}{3}P&\frac14P_{\bX}\\
0& 0& \frac{1}{3}P
\end{array}\right]+\bl\left[ \begin{array}{ccc}
-\frac{1}{3}\bP&\frac14P_{\bX}&\frac1{12}(-\bP P_{\bX}-4\bP_{Y}+P_{\bX\bX}) \\
0&-\frac{1}{3}\bP&0\\
0& 1& \frac{2}{3}\bP
\end{array}\right].
\end{align*}}

A full Cartan connection $\om_C$ on a global trivialization $$(x,b)\in M_\C\times B_\C$$ of $\ba_\C$ is obtained by $$\om_\C(x,b)=Ad(b^{-1})s^*\om_\C(x)+\om_B(b),$$ where $\om_B$ is the Maurer--Cartan form on $B_\C$, c.f. \cite[Section 1.5.4]{P}. To make sure that the normalization condition of \cite[Section 3.1.12]{P} are indeed satisfied in our situation, it suffices to consider a general pullback $$s'(x)=s(x)\phi(x)$$ for $\phi: M \to B_\C$ and check that he obtains the same classifying sets for $$s'^*\om=Ad(\phi(x)^{-1})s^*\om+\delta\phi,$$ where $\delta$ is the left logarithmic derivative of $\phi$. 

\subsection{Construction of the invariants}

If one evaluates the structure equation $d\omega_\C+[\omega_\C,\omega_\C]$ on the constant vector fields $\om_\C^{-1}(X),\om_\C^{-1}(Y)$ for $X,Y\in \frak{sl}(3,\C)$, then one obtains the curvature function $$\kappa: \ba_\C\to \wedge^2 (\frak{sl}(3,\C)/\fb_\C)^*\otimes \frak{sl}(3,\C)$$ of the Cartan geometry $(\ba_\C\to M_\C,\om_\C)$ of type $(PGl(3,\C),B_\C)$. The curvature function $\kappa$ is $B_\C$--equivariant, where the action of $B_\C$ on $\wedge^2 (\frak{sl}(3,\C)/\fb_\C)^*\otimes \frak{sl}(3,\C)$ is induced by the adjoint action on each component. This means that it suffices to compute the pullback $$s^*\kappa: M_\C\to \wedge^2 (\frak{sl}(3,\C)/\fb_\C)^*\otimes \frak{sl}(3,\C).$$ Since the pullback $s^*\kappa$ is given by evaluating $ds^*\omega+[s^*\omega,s^*\omega]$ on the projections of the restrictions of constant vector fields to image of $s$, we get the following formula for  $s^*\kappa$:
{\tiny
\begin{align*}
s^*\kappa(\left[ \begin{array}{ccc}
0& 0&0 \\
x&0&0\\
y& \bar x&0
\end{array}\right],
\left[ \begin{array}{ccc}
0& 0&0 \\
u&0&0\\
v& \bar u&0
\end{array}\right])=(xv-yu)\left[ \begin{array}{ccc}
0&I_1&I_2\\
0&0&0\\
0& 0& 0
\end{array}\right]+(\bar xv-y\bar u)\left[ \begin{array}{ccc}
0&0&I_4\\
0&0&I_3\\
0& 0& 0
\end{array}\right],
\end{align*}}
where
\begin{align*}
I_1&=\frac16(-P_{XX\bX}+3P P_{X\bX}+2P_{YX}+P_{X}P_{\bX}-2P^2P_{\bX}-2PP_{Y})\\
I_2&=(I_1)_X\\
I_3&=\frac16(P_{\bX\bX\bX}+2\bP_{Y\bX}-3P_{\bX\bX}\bP-\bP_{\bX}P_{\bX}+2\bP^2P_{\bX}-2\bP \bP_{Y})\\
I_4&=(I_3)_{\bX}.
\end{align*}
Note, that the differential relations between the invariants are a consequence of Bianchi identity, c.f. \cite[Section 1.5.9]{P}.

Further, one considers the iterated derivatives of $\kappa$ in directions of constant vector fields. In other words, the remaining invariants are function 
\[\ba_\C\to\bigotimes_{i\geq 0}^i\frak{sl}(3,\C)^*\otimes \wedge^2 (\frak{sl}(3,\C)/\fb_\C)^*\otimes \frak{sl}(3,\C),\]
 which are again $B_\C$--equivariant. The pullback of derivative in the direction of constant vector field $$\om_\C^{-1}(X)=\xi+\tau$$ for $\xi$ tangent to image of $s$ and $\tau$ vertical is derivative in direction $s^*\xi$ of the projection of $\xi$ of the pullback minus pullback of the (adjoint) action of $\om_\C(\tau)$. In particular, $$\om^{-1}(X).\kappa=s^*\xi. s^*\kappa-Ad(\om(\tau)).s^*\kappa$$ for the induced adjoint action of $\fb_{\C}$ on $\wedge^2 (\frak{sl}(3,\C)/\fb_\C)^*\otimes \frak{sl}(3,\C)$ and we obtain the following formula:

\begin{align*}
&D_{\left[ \begin{array}{ccc}
0& 0&0 \\
x&0&0\\
y& \bar x&0
\end{array}\right]+b}\kappa(U,V)=x\kappa(U,V)_X+\bar x\kappa(U,V)_{\bX}+y\kappa(U,V)_Y\\
&-(-xs^*\om(X)-\bar xs^*\om(\bX)-ys^*\om(Y)+b+\left[ \begin{array}{ccc}
0& 0&0 \\
x&0&0\\
y& \bar x&0
\end{array}\right]).\kappa(U,V),
\end{align*}
where $x,\bar x,y\in \C$ and $b\in \fb_\C$ and $.$ is the natural action of $\fb_\C$ on $\wedge^2 (\frak{sl}(3,\C)/\fb_\C)^*\otimes \frak{sl}(3,\C)$ induced by adjoint action. This formula naturally extends to iterated fundamental derivatives.

Then components of $\kappa$ and its derivatives  (of order less than $k$) on $\mathcal{Y}=\ba_\C$ provide all invariants (of order less than $s$). Among them, we can pick some functionally independent set $P_1,\dots,P_n$ of invariants and test the obstruction of Huang, Ji and Yau.

\subsection{Algebraicity of invariants in the Main Theorem}

Consider a rigid hypersurface $$\Re w=H(z,\bz)$$ in $\C^2$ satisfying that $H_{z \bz}$ is algebraic and everywhere nonzero. If we look at the computations in the previous section we see that the function $P$ and all its derivations are algebraic and thus all invariants are algebraic. This implies the first part of the Main Theorem.

In the case of example \eqref{exmp1}, we compute that $$H_{z \bz}=1+\frac{1}{1+z^2}+\frac{1}{1+\bz^2}$$ is algebraic and thus all the invariants computed as above are algebraic and nonzero as $$H_{zz \bz}=\frac{-2z}{(1+z^2)^2},\quad H_{z\bz \bz}=\frac{-2\bz}{(1+\bz^2)^2},\quad H_{zz \bz\bz}=0$$ lead to
{\tiny
\[
I_1={\frac {8{ z}{ \bz} \left( 3{{ z}}^{6}{{
 \bz}}^{4}+12{{ z}}^{6}{{ \bz}}^{2}-6{{ z}}^{4}{{ 
\bz}}^{4}+12{{ z}}^{6}-15{{ z}}^{4}{{ \bz}}^{2}-10{{
 z}}^{2}{{ \bz}}^{4}-6{{ z}}^{4}-36{{ z}}^{2}{{ 
\bz}}^{2}-2{{ \bz}}^{4}-29{{ z}}^{2}-11{{ \bz}}^{2}-12
 \right) }{3 \left( {{ z}}^{2}+1 \right) ^{2} \left( {{ z}}^{2}
{{ \bz}}^{2}+2{{ z}}^{2}+2{{ \bz}}^{2}+3 \right) ^{4}}}.
\]}
We can compute also the further invariants that are also  algebraic, i.e., the obstructions of Huang, Ji and Yau are satisfied. However, they are too complicated to be presented explicitly.

\section{Proof of the Main Theorem}

%
%

We are now in the position to prove our main result.

\begin{proof}[Proof of Main Theorem]
The first assertion of the Main Theorem is addressed in Section 4.3

For the second assertion, we argue by contradiction and assume $M$ is algebraizable. Let us return to the associated ODE procedure in Section 2.1. Note that the hypersurface $M$ is rigid, thus for the associated ODE, 
$$\Phi(z,w,w\rq{})=\Phi(z,w\rq{}),$$ and for the function $A$, as in \eqref{AB}, $$A(z,w,w\rq{})=A(z,w\rq{}).$$
Since $H_z(z,\bar z)$ is non-algebraic in $\bar z$, the above function $A(z,w\rq{})$ is also non-algebraic in $w\rq{}$ (otherwise $H_z$ would need to be algebraic in $\bar z$). Now, we have 
\beq\label{identity}
\Phi(z,w\rq{})=H_{zz}\bigl(z,A(z,w\rq{})\bigr).
\eeq
In view of \autoref{obstruct}, $\Phi(z,w\rq{})$ is algebraic in $w\rq{}$. Recall finally that $H_{zz}$ is algebraic in $\bar z$.  Now, solving \eqref{identity} for $A$ (this is possible in view of $H_{zz\bar z}\not\equiv 0$), we conclude that $A(z,w\rq{})$ has to be algebraic in $w\rq{}$, which is a contradiction. The theorem is proved.

\end{proof}

For example, $$w\rq{}=\bz+\frac{\bz}{1+z^2}+\arctan{\bz}$$ is not algebraic in the case of example \eqref{exmp1}, but $$w\rq{}\rq{}=\frac{-2z\bz}{(1+z^2)^2}$$ is algebraic. Note that $I_1$ and further invariants will not be algebraic in the $(z,w\rq{})$ coordinates, but the relations between the invariants remain unchanged and thus this is an counterexample for the obstruction of Huang, Ji and Yau to be sufficient for algebraizibility.

\appendix

\section{The Cartan connection $\om$ of type $(PSU(1,2),B)$}\label{appB}

Let us recall, how to obtain the real form $(\ba\to M,\om)$ from the Cartan connection $\om_\C$ we constructed in Section \ref{Expl}.

Firstly, if we consider the pointwise complexifications of the linear isomorphisms $\om(u)$ for all $u\in \ba$, then they are obtained precisely by the formulas in Section \ref{Expl}, because the arguments of the functions in the formula play no role (if one interprets the $\partial_z$ and $\partial_{\bz}$ in the usual way). In particular, the formulas are valid without assumption that defining functions are real analytic, but it suffices that all the derivatives in the formulas are defined.

To get $\om$, we need to choose the realization of the real form $\frak{su}(1,2)\subset \frak{sl}(3,\C)$. For this, it suffices to choose that $l$ and $\bl$ are conjugated (as the notation suggests) and that $j$ is purely imaginary. This way we obtain that the real form is realized $\frak{su}(1,2)$ by the following matrices:
$$\left[ \begin{array}{ccc}
x_4+ix_5& x_6+ix_7&ix_8 \\
x_1+ix_2&-2x_5&x_6-ix_7\\
ix_3& x_1-ix_2&-x_4+ix_5.
\end{array}\right]$$
Let us emphasize that since frame $(Y,X,\bX)$ is not holonomic and $P_{\bX}$ is real, it is not on first sight visible that the formulas in Section \ref{Expl} are compatible with this real form. For example, if the reader would like to check that the invariants $I_1$ and $I_3$ are conjugated, then he needs to rearrange the derivations after conjugation using the Jacobi identity to get the same formulas.

\section{Coframe for general hypersurfaces in $\C^2$}\label{Ap1}

Let us consider a strongly pseudo--convex real-analytichypersurface given by defining equation $u=F(z,\bz,v),$ where $u=\Re w$ and $v=\Im w$. In the setting of Section \ref{Expl}, we start we the following frame:
\begin{align*}
X&=\partial_z+A\partial_v,\\
\bX&=\partial_{\bz}+\bar A\partial_v\\
Y&=-(\bar A_z-A_{\bz}+A\bar A_v-\bar A_v)\partial_v,
\end{align*}
where $A:=-i\frac{F_z}{1+iF_v}$ with the arguments of $F$ omitted and the subscript used for the partial derivatives. Observe, that $X,\bar X$ generate the complexification of the complex tangent space and that $[X,\bar X]=-Y$, which provides the required Heisenberg Lie algebra complement to $\fb_\C$ in $\frak{sl}(3,\C)$.

Therefore, we use the corresponding coframe $(j,l,\bl)$ dual to the frame $(Y,X,\bX)$ and compute that
\begin{align*}
dj &=P j\wedge l+\bP j\wedge \bl+l\wedge \bl,\\
 dl &=0,\\
  d\bl&=0,\\
P\cdot Y&:=-[X,Y]=\frac{(\bar A_{vv}A-A_{vv}\bar A+2\bar A_{zv}-A_{\bz v})A}{\bar A_{z}-A_{\bz}+A\bar A_{v}-\bar AA_{v}}+\\
&+\frac{(A_{\bz}-A\bar A_{v}+\bar AA_{v}-2\bar A_{z})A_{v}+A_{z}\bar A_{v}-\bar AA_{zv}+\bar A_{zz}-A_{z\bz}}{\bar A_{z}-A_{\bz}+A\bar A_{v}-\bar AA_{v}}\cdot Y. 
 \end{align*}
Note that a consequence of Jacobi identity is again $P_{\bX}=\bP_X$.

The Cartan connection $\om_\C$ and the corresponding invariants is then obtained by the formulas in Section \ref{Expl}. We emphasize that the formulas take into account that the frame $(Y,X,\bX)$ is not holonomic.


\begin{thebibliography}{99}


\bibitem[BER99]{ber}  M. S. Baouendi, P. Ebenfelt, L. P. Rothschild.
Real Submanifolds in Complex Space and Their Mappings.
Princeton University  Press, Princeton Math. Ser. {\bf 47},
Princeton, NJ, 1999.

\bibitem[BK89]{bluman} G.~Bluman and S.~Kumei. Symmetries and differential equations. Applied Mathematical Sciences, 81. Springer-Verlag, New York, 1989.

\bibitem[CS09]{P}   \v Cap, A.,  Slov\'ak, J., Parabolic Geometries I:
Background and General Theory, Amer. Math. Soc., 2009.

\bibitem[Ca32]{C32} Cartan, E., {\it Sur la géométrie pseudo-conforme des hypersurfaces de deux variables com- plexes, I}. Ann. Math. Pura Appl. (4) Il (1932), 17-90, (or Oeuvres II, 2, 1231-1304); II, Ann. Scuola Norm. Sup. Pisa (2) 1 (1932), 333-354 (or Oeuvres III, 2, 1217-1238).

\bibitem[Ch75]{C75} Chern, S. S., {\it On the projective structure of a real hypersurface in $\CC{n+1}$}, Math. Scand. 36 (1975), 581–600.

\bibitem[CM74]{CM74} Chern, S. S. and Moser, J. K., {\it Real hypersurfaces in complex manifolds,} Acta Math. 133 (1974), 219–271.

\bibitem[Fa80]{F80} Faran, J., {\it Segre families and real hypersurfaces}, Invent. Math. 60 (1980), 135–172.

\bibitem[Fo04]{forstnericalg} F. Forstneric, {\it Most real-analytic Cauchy-Riemann manifolds are nonalgebraizable}, Manuscripta Math. {\bf 115} (2004), no. 4, 489-494.

\bibitem[HJY01]{HJY01} Huang, X., Ji, S., Yau, S. S. T., {\it An example of a real-analytic strongly pseudoconvex hypersurface which is not holomorphically equivalent to any algebraic hypersurface}. Ark. Mat. 39 (2001), no. 1, 75--93.

\bibitem[KL18]{nonanalytic} I.\,Kossovskiy,  B.\,Lamel. {\it On the analyticity of CR-diffeomorphisms.} Amer. J. Math. 140 (2018), no. 1, 139--188.

\bibitem[KS16a]{divergence} I. Kossovskiy and R. Shafikov. 
{\it Divergent CR-equivalences and meromorphic differential
equations.} J.~Eur. Math. Soc. (JEMS) 18 (2016), no. 12, 2785--2819.

\bibitem[KS16b]{nonminimalODE} I. Kossovskiy and R. Shafikov. {\it Analytic
Differential Equations and  Spherical Real Hypersurfaces.} J.~Differential Geom. 102 (2016), no. 1, 67--126.

\bibitem[La02]{lang} S.Lang. Algebra. Graduate Texts in Mathematics, 211. 2002.

\bibitem[Ol95]{O95} Olver, P.J., Equivalence, Invariants and Symmetry, Cambridge University Press, Cambridge, UK, 1995

\bibitem[Se32]{segre} B. Segre. {\em Questioni geometriche legate colla teoria delle funzioni di due variabili
complesse.} Rendiconti del Seminario di Matematici di Roma, II,
Ser. 7 (1932), no. 2, 59-107.

\bibitem[Su01]{sukhov1}  A. Sukhov. {\em Segre varieties and Lie symmetries.}
Math. Z. {\bf 238} (2001), no. 3, 483-492.

\bibitem[Su03]{sukhov2} A. Sukhov {\em On transformations of analytic
CR-structures.} Izv. Math. 67 (2003), no. 2, 303--332

\bibitem[Ta62]{T62} Tanaka, N., I. {\it On the pseudo-conformal geometry of hypersurfaces of the space of n complex variables}, J. Math. Soc. Japan14, 397–429 (1962); II. Graded Lie algebras and geometric structures, Proc, U.S.-Japan Seminar in Differential Geometry, pp. 147–150, 1965

\bibitem[Ta75]{T75} Tanaka, N., {\it A differential geometric study on strongly pseudoconvex manifolds,} Lectures in Math., vol. 9, Kyoto Univ., 1975.

\bibitem[Ta96]{T96} Tresse, A.: Détermination des invariants ponctuels de l'équation differentielles ordinaire du second ordre $y''=\om(x, y, y')$, Preisshr. Fürstlich Jablon Ges, Leipzig: Hirzel 1896


\bibitem[We77]{websteralg}  S. Webster. {\em On the mappings problem for algebraic
real hyprsurfaces}, Invent. Math., {\bf 43} (1977), no. 1, 53-68.



\bibitem[Za08]{zaitsevquad} D. Zaitsev, {\it Obstructions to embeddability into hyperquadrics and explicit examples}, {\it Math. Ann.}, {\bf 342} (2008), no. 3, 695-726.

 \bibitem[Za99]{zaitsevalg}  D.~Zaitsev. {\it Algebraicity of local holomorphisms between real-algebraic submanifolds of complex spaces.}  Acta Math. 183 (1999), 273--305.



\end{thebibliography}
\end{document}